\documentclass[manyauthors]{fundam}

\usepackage{url} 
\usepackage[ruled,lined]{algorithm2e}
\usepackage{graphicx}
\usepackage{tikz}
\usetikzlibrary{automata, positioning, arrows}

\usepackage{hyperref}

\begin{document}

\setcounter{page}{285}
\publyear{2021}
\papernumber{2074}
\volume{182}
\issue{3}

 \finalVersionForARXIV

\title{Edge Forcing in Butterfly Networks}

\author{Jessy Sujana G. \\
Department of Computer Science and Engineering \\
Sri Sivasubramaniya Nadar College of Engineering\\
Chennai,  603 110, India \\
jessysujanag@ssn.edu.in
\and T.M. Rajalaxmi\thanks{Address for correspondence: Department of Mathematics, Sri Sivasubramaniya Nadar College
                   of Engineering, Chennai, 603 110, India. \newline \newline
          \vspace*{-6mm}{\scriptsize{Received  July 2021; \ revised July 2021.}}}
\\
Department of Mathematics \\
Sri Sivasubramaniya Nadar College of Engineering \\
Chennai, 603 110, India\\
laxmi.raji18@gmail.com \\
\and Indra Rajasingh\\
School of Advanced Sciences \\
Vellore Institute of Technology\\
Chennai,  600 127, India\\
indra.rajasingh@vit.ac.in\\
\and R. Sundara Rajan\\
Department of Mathematics \\
Hindustan Institute of Technology and Science \\
Chennai, 603 103, India\\
vprsundar@gmail.com
}

\maketitle

\runninghead{G. Jessy Sujana et al.}{Edge Forcing in Butterfly Networks}

\begin{abstract}
  A zero forcing set is a set $S$ of vertices of a graph $G$, called forced vertices of $G$, which are able to force the entire graph by applying the following process iteratively: At any particular instance of time, if any forced vertex has a unique unforced neighbor, it forces that neighbor. In this paper, we introduce a variant of zero forcing set that induces independent edges and name it as edge-forcing set. The minimum cardinality of an edge-forcing set is called the edge-forcing number. We prove that the edge-forcing problem of determining the edge-forcing number is NP-complete. Further, we study the edge-forcing number of butterfly networks. We obtain a lower bound on the edge-forcing number of butterfly networks and prove that this bound is tight for butterfly networks of dimensions 2, 3, 4 and 5 and obtain an upper bound for the higher dimensions.
\end{abstract}


\section{Introduction}

A propagation model can be related to an activation process in a graph. By iteratively applying an activation rule and by using an initial set of active vertices, the remaining vertices are activated in a graph. This process ends when there are no more vertices to be activated. This paper aims to find the minimum size of a set of such initially active vertices in a graph with a well defined forcing rule.

\medskip
The zero forcing problem states that, for graph $G$, the goal is to find a minimum set of nodes $S$ that forces all the other nodes, where a node $v$ is forced if and only if, $v$ is an element of the set $S$ or $v$ has a neighbor $u$ such that $u$ and all of its neighbors except $v$ are forced. The minimum cardinality of any such set is called zero forcing number of $G$ {\rm\cite{benson2015}}.

\medskip
Equivalently, we have the following definition:

\begin{definition} {\rm\cite{Daniela2018}}  For a graph $G=(V,E)$ and a set $T \subseteq V$, the closure of $T$ in $G$ denoted by $C_{G}(T)$is recursively defined as follows: Start with $C_G(T) = T$. As long as exactly one of the neighbors of some element of $C_G(T)$ is not in $C_G(T)$, add that neighbor to $C_{G}(T)$. If $C_G(T)=V$ at some stage, then $T$ is a zero forcing set of $G$. A forcing set of minimum cardinality is called the forcing number and is denoted by $\zeta(G)$. The forcing process is also called Graph Infection or Graph Propagation.
\end{definition}

Zero forcing number is a graph parameter introduced as a tool for solving minimum rank problem {\rm\cite{aim2008}. The idea of zero forcing set which is also termed ``infecting set" was introduced in 2007 by {\rm \cite{bg2007}} and {\rm \cite {bm2009}} in relation to quantum systems. The zero forcing problem of determining the forcing number is NP-complete {\rm\cite{shaun2016}}.  It is also used in theoretical computer science as a fast mixed search model {\rm\cite{Daniela2018}}. Zero forcing number was obtained in different types of graphs namely generalised Petersen graphs {\rm\cite{Saeedeh2019}}, cacti graphs {\rm\cite{Darren2011}}, graphs of large girth {\rm\cite{Randy2015}}, fixed bipartite graphs, random and pseudo-random graphs {\rm\cite{thomas2019}}. In addition to this, zero forcing was effected in snake graphs {\rm\cite{Anitha2019}}, wheel graphs {\rm\cite{Linda2012}}, fan graphs, friendship graphs, helm graphs {\rm\cite{Sakander2020}} and generalised Sierpinski graphs {\rm\cite{Ebrahim2019}}. An upper bound for zero forcing number of butterfly networks of dimension $r$ has been obtained as  $ \dfrac{1} {9} [ (3r+7)2^{r} + 2(-1)^{r}]$ in {\rm \cite{sudeep2017}}. The influence of removing a vertex or edge on the zero forcing number was studied and the propagation time for zero forcing on a graph was also determined in {\rm\cite{Leslie2012}}.}

\medskip
In electrical power systems, where an electrical node is represented as a vertex and a transmission line joining two electrical nodes by an edge, Phase Measurement Units (PMUs) are placed at selected vertices to regularly access or monitor electrical parameters like phase and voltage. Due to the high cost of a PMU, placing them at the locations of a minimum zero forcing set of the system, helps the monitoring of the entire system.

\medskip
Variants of forcing such as total forcing {\rm\cite{Michael2019}, connected forcing {\rm\cite{sudeep2017, Haynes2002, thomas2019, colton2018}} and $k$-forcing {\rm\cite{Daniela2011, bg2007, david2014}} have been considered by several authors. The total forcing problem has been proved to be NP-complete {\rm\cite{randy2017}}}.

\medskip
In this paper, we introduce a new problem called Edge-Forcing problem and it is defined as follows:

\begin{definition} Let $G$ be a graph. For a set $K$ of independent edges in $G$, define $T$ $(T$ depends on $K)$ to be the set of all end points of edges in $K$. The closure of $T$ in $G$ denoted by $C_{G}(T)$ is recursively defined as follows: Start with $C_G(T) = T$. As long as exactly one of the neighbors of some element of $C_G(T)$ is not in $C_G(T)$, add that neighbor to $C_{G}(T)$. If at some stage, $C_{G}(T) = V$, then $K$ is called an edge-forcing set of $G$. The minimum cardinality of an edge-forcing set of $G$ is the edge-forcing number of $G$ and is denoted by $\zeta_{e}(G)$. The edge-forcing problem of a graph $G$ is to  \linebreak determine $\zeta_{e}(G)$.
\end{definition}\vspace*{-2mm}

The above problem may also be viewed as coloring of vertices originating from a set of independent edges. Let $G$ be a graph in which every vertex is initially colored either black or white with at least one edge with both ends colored black. Let $e=(u,v)$ be an edge such that both $u$ and $v$ are colored black. If $u$ or $v$  is adjacent to exactly one white neighbor, say $x$, then we change the color of $x$ to black; this rule is called the color change rule. In this case we say ``$e$ forces $x$" which is denoted by $e \rightarrow x$. At a time, $e$ may force two vertices. The procedure of coloring a graph using the color change rule is called a forcing process. Given an initial coloring of $G$ in which a set of vertices inducing a set of independent edges is black and all other vertices are white, the derived set is the set of all black vertices resulting from repeatedly applying the color change rule until no more changes are possible. If the derived set from an independent set of edges is the entire vertex set of the graph, then the set of initial edges is called an edge-forcing set. It is also addressed as \linebreak $P_{2}$-forcing set.

\medskip
In Figure~\ref{figure1}, the edge marked in red in graph $G$ is an edge-forcing set. At the first time step, vertex 2 forces vertex 4, at second time step, vertex 1 forces vertex 3, at third time step, vertex 3 or 4 forces vertex 5, in the fourth time step, vertex 5 forces vertex 6.

\begin{figure}[h!]
\begin{center}
\includegraphics[scale=0.77]{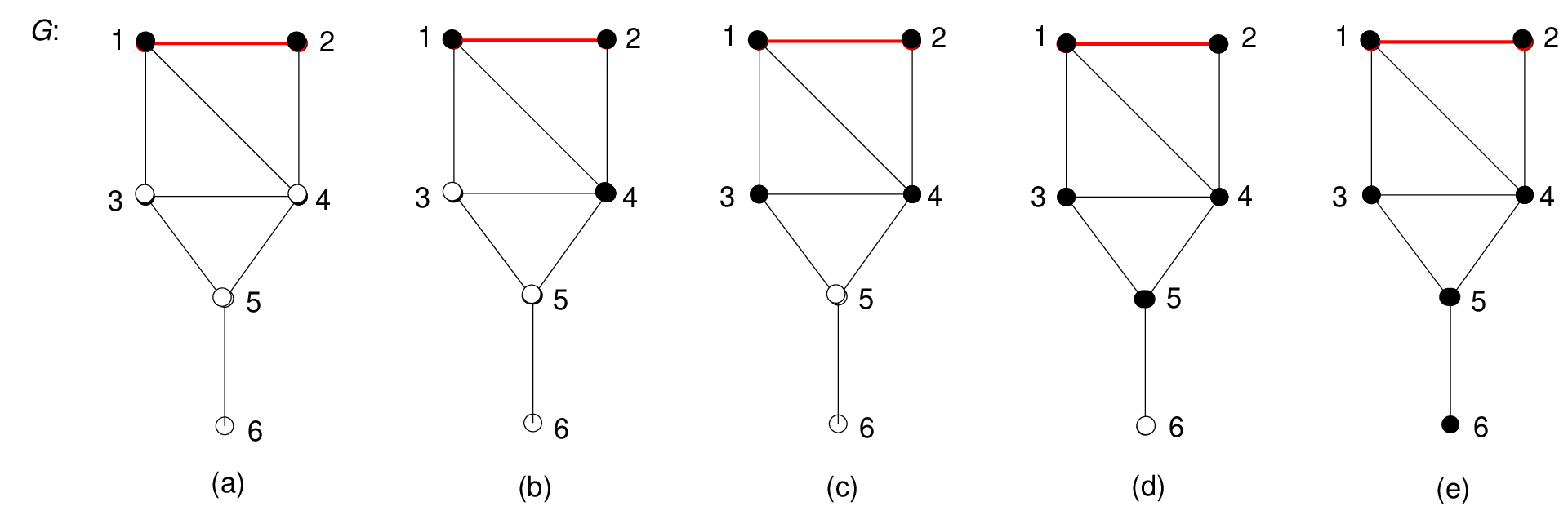}\vspace*{-5mm}
\end{center}\vspace*{-3mm}
\caption{Illustration of color change in $G$, with $\zeta_{e}(G)$=1. (a) $K=\{(1,2)\}$ with vertices 1 and 2 colored black, (b) 2 forces 4, (c) 1 forces 3, (d) 3 forces 5, (e) 5 forces 6 }
\label{figure1}
\end{figure}

Edge-forcing set ensures more reliability in the system. For example, the PMUs placed adjacent to each other in an electrical system can be used as backup servers so that if one becomes faulty, the adjacent PMU can support the system, thereby monitoring the entire system without any \linebreak  interruption.

\section{Complexity of edge-forcing problem}

In this section we prove that edge-forcing problem is NP-complete. The reduction will be from the NP-completeness of zero forcing problem {\rm\cite{shaun2016}}. Let $G=(V,E)$ be a graph. For $x \in V$, let $N(x)$ denote the open neighborhood of $x$ in $G$. Then construct the graph $\bar{G}=(\bar{V}, \bar{E})$ as follows. The vertex set $\bar{V}=V \cup V'$, where $V'=\{x':x \in V\}$. The edge set $\bar E=E\cup E' \cup E''$, where $E'=\{(x,x'): x\in V$ and $x' \in V' \}$ and $E''=\{(y,x'): y\in N(x), x, y \in V\}$. See Figure~\ref{figure2}.

\begin{figure}[h!]
\begin{center}
\includegraphics[scale=0.95]{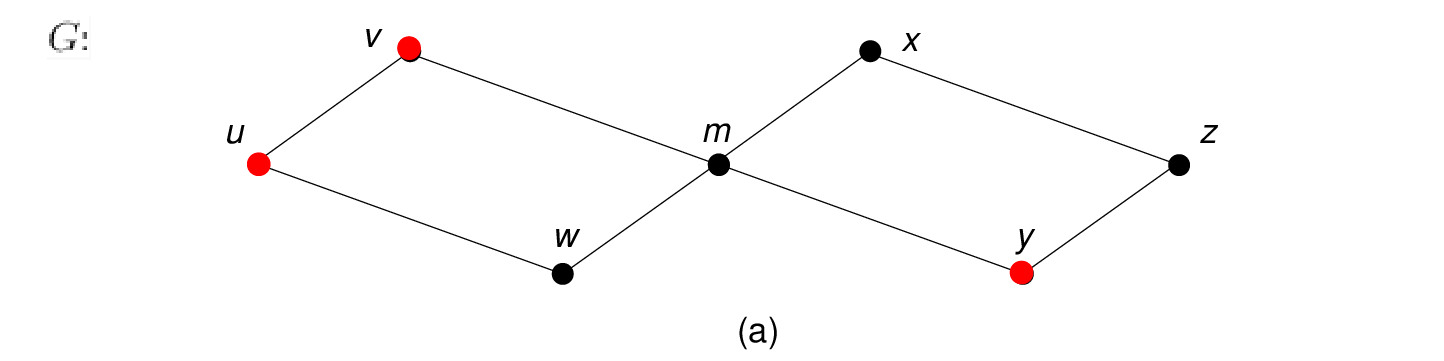}
\end{center}
\end{figure}
\begin{figure}[h!]
\vspace*{-5mm}
\begin{center}
\includegraphics[scale=0.9]{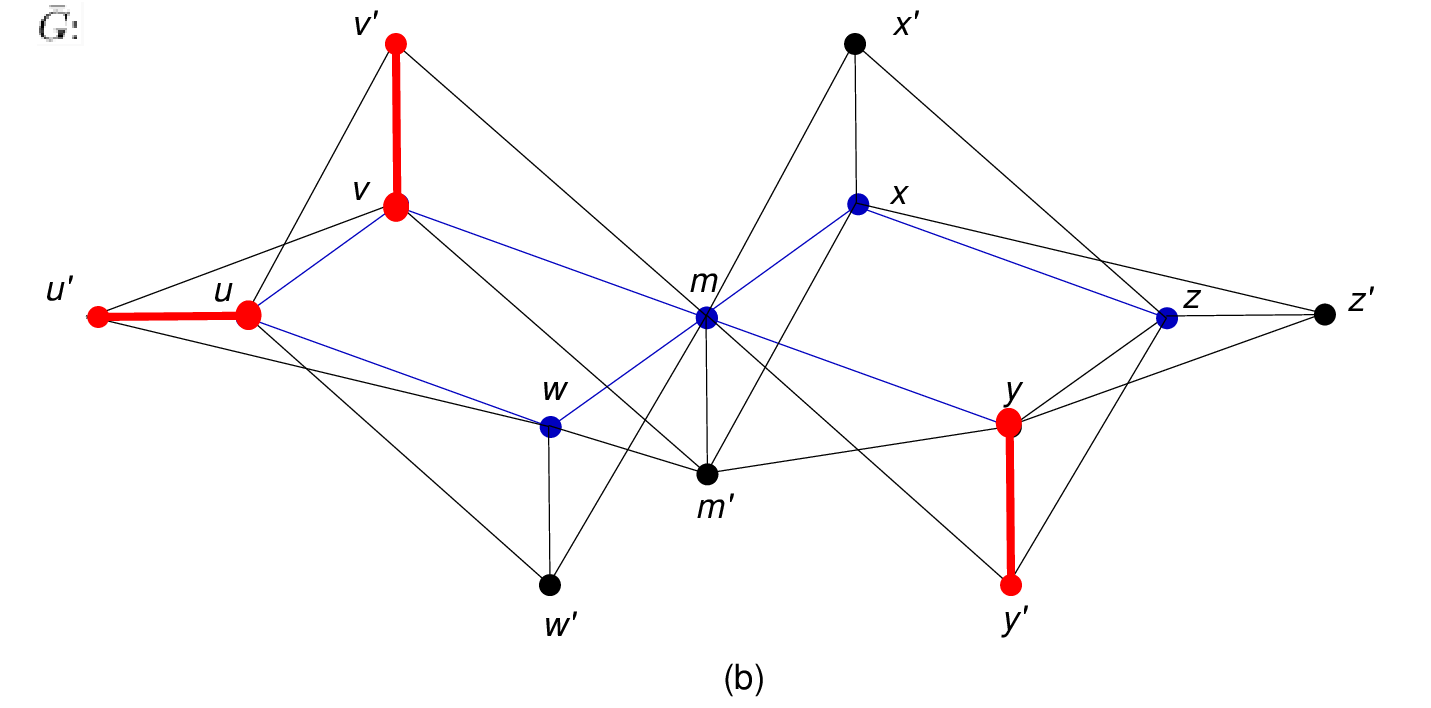}
\end{center}\vspace*{-6mm}
\caption{(a) Graph $G=(V,E)$, with zero forcing set marked in red, (b) $\bar G=(\bar V,\bar E)$, with edge-forcing set marked in red }
\label{figure2}\vspace*{-3mm}
\end{figure}

\begin{lemma}
If $X$ is an edge-forcing set of $\bar G$, then there exists an edge-forcing set $Y$ of $\bar G$ with $|Y|=|X|$, such that $Y\subseteq E'$.
\end{lemma}

\begin{proof}We have the following two cases.
\\
Case 1: If $X\subseteq E'$, then $Y=X$.\\
Case 2: If $X\not\subseteq E'$, then $X$ may contain edges from $E$ as well as $E''$. Suppose $X$ contains an edge $(u,v') \in E''$. Then replace the edge $(u,v')$ in $X$ with $(u,u')$, since the vertices $\{u,u'\}$ force all the vertices which the vertices $\{u,v'\}$ would force. Suppose $X$ contains an edge $(a,b) \in E$. Then replace the edge $(a,b)$ in $X$ with either $(a,a')$ or $(b,b')$, since the vertices $\{a,a'\}$ or $\{b,b'\}$ force all the vertices which the vertices $\{a,b\}$ would force. The final set of edges thus obtained on replacing edges of $E\cup E''$ with the respective edges in $E'$, is taken as $Y$. Clearly $Y$ is an edge-forcing set of $\bar G$  with $|Y|=|X|$ and $Y\subseteq E'$.
\end{proof}

Now we can prove the essential part for our reduction.

\begin{lemma}
Every zero forcing set of $G$ induces an edge-forcing set of $\bar G$ of the same cardinality and conversely.
\end{lemma}

\begin{proof}
Suppose $S$ is a zero forcing set of $G$. Considering the edge set $S'=\{(x,x'): x\in S\}$ in $\bar G$, if $u\in S$ forces a vertex $v\in V\backslash S$ in $G$, then in $\bar G$ initially $u'$ forces $v$, after which $u$ forces $v'$. Thus the vertices $\{u,u'\}\in S'$ force the vertices $\{v,v'\}\in V'\backslash S'$ in $\bar G$. It is clear that vertices of set $S'$ iteratively force all the vertices of $\bar G$. Hence $S'$ is an edge-forcing set of $\bar G$.

\medskip
Conversely, suppose $X$ is an edge-forcing set of $\bar G$. By Lemma 2.1 we can obtain a set $S'\subseteq E'$ which is an edge-forcing set of $\bar G$. Consider vertex set $S=\{x: (x,x')\in S'\}$  in $G$. If the vertices $\{u,u'\}\in S'$ force the vertices $\{v,v'\}\in V'\backslash S'$ in $G$, it is evident that $u\in S$ forces the vertex $v\in V\backslash S$ in $G$. Hence $S$ is a zero forcing set of $G$.
\end{proof}
Note that Lemma 2.2 also implies that $S\subseteq V$ is a minimum zero forcing set of $G$ if and only if $S'=\{(x,x') : x\in S\}$ is a minimum edge-forcing set of  $\bar G$. Combining this result with the fact that the graph $\bar G$ can clearly be constructed from $G$ in polynomial time, we arrive at the following main result.

\begin{theorem}
The edge-forcing problem is NP-complete.
\end{theorem}

\section{Butterfly networks}
Interconnection network is a connection pattern of the components in a system. This is necessary for fast and trusted communication among systems in any parallel computer. The development of large scale integrated circuit technology has led to the growth of complex interconnection networks {\rm \cite{Konstantinidou1993}}. Graph Theory is used in the analysis and design of these complex networks {\rm \cite{SG2010}}.

\medskip
An interconnection network is characterised as a model in graph theory where vertices and edges are represented by devices and their communication links  respectively. In a multistage network, $N$ inputs are connected to $N$ outputs. The butterfly network is an important multistage interconnection network which has an attractive architecture for communication {\rm \cite{Paul2013}}. It overcomes few of the disadvantages of the hypercube, which is the standard network used in industries. Some of its advantages are high bandwidth, small diameter and constant degree switches. Butterfly networks are used to perform a method to demonstrate fast fourier transform used in the area of signal processing {\rm \cite{Indra2011}}.

\begin{definition} {\rm \cite{Indra2011}}
The $r$-dimensional butterfly network $BF(r)$, has a vertex set
$V=\left\{\left[\left(x_{1},x_{2},\ldots,\right.\right.\right.$ $\left.x_{r}\right); \left.i\right]/x_{i}=\left.0\ or\ 1,\ 0 \leq i\leq r\right\}$ and\ vertices $\left[\left(x_{1},x_{2},\ldots,x_{r}\right);i\right]$ and $\left[\left(y_{1},y_{2},\ldots,y_{r}\right);j\right]$ are connected if $ j=i+1$ and either $\left(x_{1},x_{2},\ldots,x_{r}\right)=\left(y_{1},y_{2},\ldots,y_{r}\right)$ or $\left(x_{1},x_{2},\ldots,x_{r}\right)$ and \linebreak $\left(y_{1},y_{2},\ldots,y_{r}\right)$ differ precisely in the $j^{th}$ bit.
\end{definition}

\begin{remark} The vertices $\left[\left(x_{1},x_{2},\ldots,x_{r}\right);i\right]$ are said to be at Level $i$, $0 \leq i\leq r$. If the end vertices of $\left[\left(x_{1},x_{2},\ldots,x_{r}\right);i\right]$ and $\left[\left(y_{1},y_{2},\ldots,y_{r}\right);i+1\right]$ of an edge satisfy the condition that $\left(x_{1},x_{2},\ldots,x_{r}\right)\linebreak = \left(y_{1},y_{2},\ldots,y_{r}\right)$, then the edge is called a straight edge. Otherwise the edge is called a cross edge. It is customary to write the binary sequence $\left(x_{1},x_{2},\ldots,x_{r}\right)$ as a decimal number.
\end{remark}

\begin{figure}[h!]
\begin{center}
\includegraphics[scale=0.58]{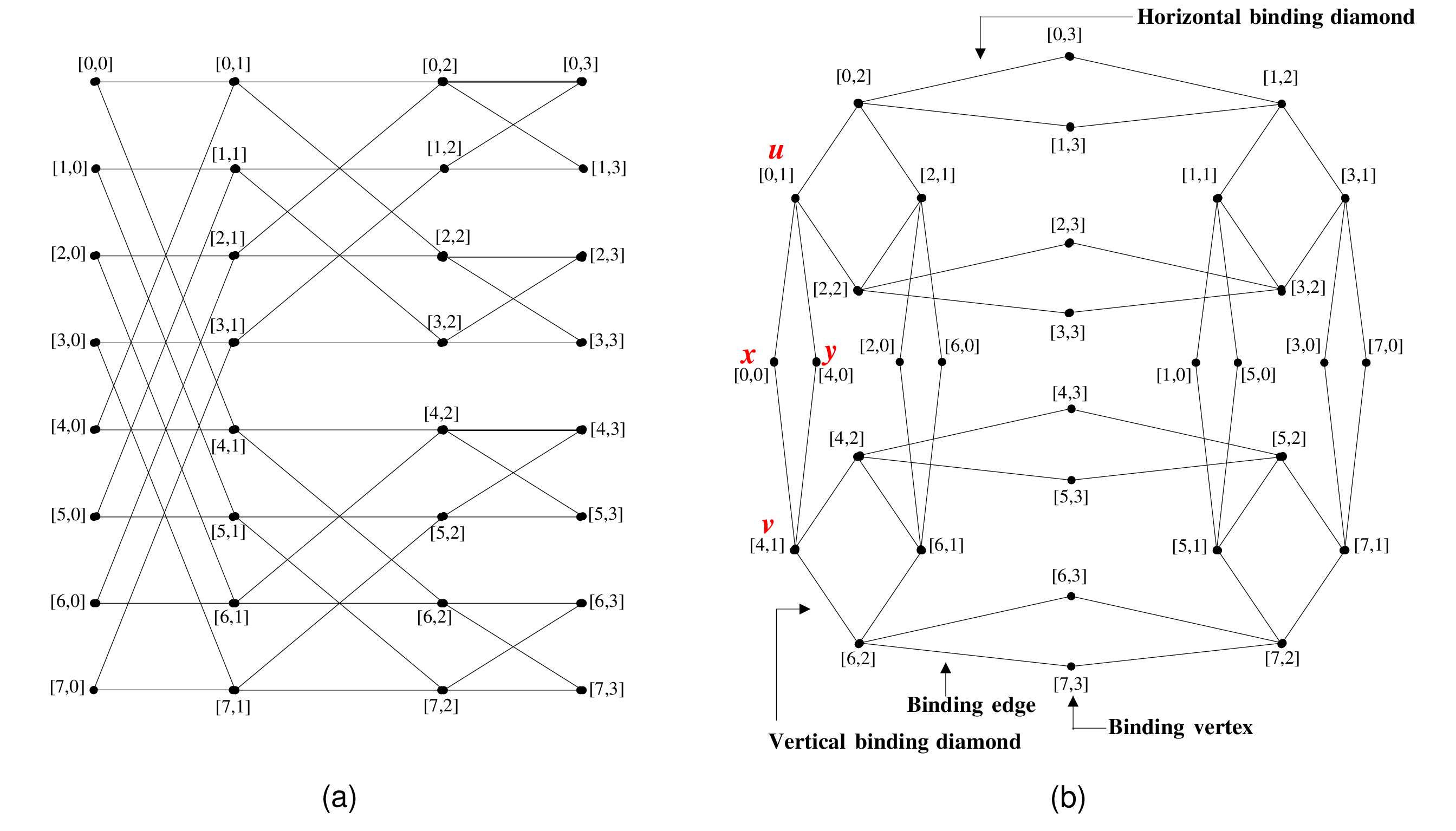}
\end{center}\vspace*{-7mm}
\caption{Butterfly network $BF(3)$; (a) Normal representation with vertices labelled using decimal numbers, (b) Corresponding diamond representation}
\label{figure3ab}\vspace*{1mm}
\end{figure}

An efficient representation for butterfly network has been obtained by Manuel et al., {\rm \cite{Paul2013}}. The butterfly network in Figure~\ref{figure3ab}(a) is the normal representation; an alternative representation, called the diamond representation, is given in Figure~\ref{figure3ab}(b). By a diamond we mean a cycle of length 4.
Two nodes $\left[ w, i \right]$ and $\left[ w', i' \right]$ in $BF(r)$ are said to be mirror images of each other if $w$ and $w'$ differ precisely in the first bit. The removal of Level 0 vertices $\left\{v_{1}, v_{2}, \ldots, v_{2^{r}} \right\}$ of $BF(r)$ gives two subgraphs $H_{1}$ and $H_{2}$ of $BF(r)$, each isomorphic to $BF(r-1)$. Since $\left\{v_{1}, v_{2}, \ldots, v_{2^{r}} \right\}$ is a vertex-cut of $BF(r)$, the vertices are called binding vertices of $BF(r)$. If a 4-cycle in $BF(r)$ has binding vertices then it is called a binding diamond. The edges of binding diamonds are called binding edges. Such diamonds are also obtained when vertices of $BF(r)$ at Level $(r+1)$ are removed. To distinguish between the two, we call the binding diamonds defined by removing the vertices at Level 0 as vertical binding diamonds and those defined by removing vertices at Level $(r + 1)$ as horizontal binding diamonds. The two types of diamonds are shown in Figure~\ref{figure3ab}(b).

\section{Edge-forcing in butterfly networks}
For determining a lower bound for the edge-forcing number in butterfly networks we consider the diamond representation of $BF(r)$, and for the actual computation of edge-forcing number, we consider the normal representation. We observe that when $r$ is even, the number of levels in $BF(r)$ is odd and vice versa.

\begin{theorem}
The edge-forcing set does not exist for $BF(2)$.
\end{theorem}

\begin{proof} There are four edge disjoint 4-cycles in $BF(2)$. If none of the edges in a 4-cycle is present in an edge-forcing set, then both vertices of degree 2 in the 4-cycle cannot be forced by any edge in the edge-forcing set. Hence every 4-cycle contributes an edge to the forcing set. See Figure~\ref{figure4}. Choice of these independent edges cannot force the left out vertices of degree 2. Hence one more edge is required in the forcing set but this induces a path of length 2 in the forcing set, a contradiction.
\end{proof}

\begin{figure}[h!]
\vspace*{-1mm}
\includegraphics[scale=0.8]{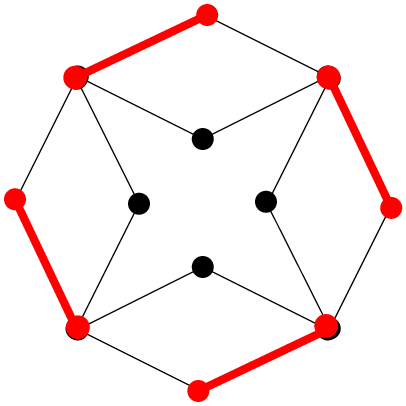}
\centering
\caption{Independent set of edges marked in red in $BF(2)$}
\label{figure4}\vspace*{-2mm}
\end{figure}

\begin{lemma}
The edge-forcing number of $BF(r)$, $r\geq 3$ is at least $2^r$.
\end{lemma}

\begin{proof}
Consider a binding diamond $D$ of $BF(r)$. Assume that none of the edges of $D$ belongs to an edge-forcing set of $BF(r)$. Let $x$ and $y$ be the binding vertices of degree 2 in $D$. Both $x$ and $y$ are adjacent to two vertices $u$ and $v$ of degree 4 in $D$. For illustration, see Figure~\ref{figure3ab}(b). Hence even if $u$ and $v$ are forced, both $x$ and $y$ cannot be forced. Hence at least one edge of $D$ should be present in the forcing set. There are $2^r$ binding diamonds in $BF(r)$, $r\geq3$. Hence $\zeta_{e}(BF(r)) \geq 2^r$.
\end{proof}

Determining $\zeta_{e}(BF(r))$, $r \geq 3$ has proved to be very challenging. In this section, $\zeta_{e}(BF(r))$ for $r=3,4$ and 5 have been determined. Using the recursive nature of $BF(r)$, an upper bound has been derived for all $r \geq 6$. It is also conjectured that this upper bound is the exact value of  $\zeta_{e}(BF(r)), r\geq 6$.

\medskip
We begin with an algorithm to determine $\zeta_{e}(BF(3))$.

\begin{flushleft} \textbf {\textit{Edge-forcing Algorithm $BF(3)$}}:\end{flushleft}
\textbf{Input}: $BF(3)$, the butterfly network of dimension 3\\
\textbf{Algorithm}:\\
Step 1: Choose the edges $([2i-1,0],[2i-1,1])$, $1 \leq i \leq 4$;\\
Step 2: Choose the edges ([2,2], [6,3]), ([3,2], [7,3]), ([0,3], [4,2]), ([1,3], [5,2]) from Levels 2 and 3.\\
\textbf{Output}: An edge-forcing set having cardinality 8

\begin{figure}[!b]
\vspace*{-1mm}
\begin{center}
\includegraphics[scale=0.7]{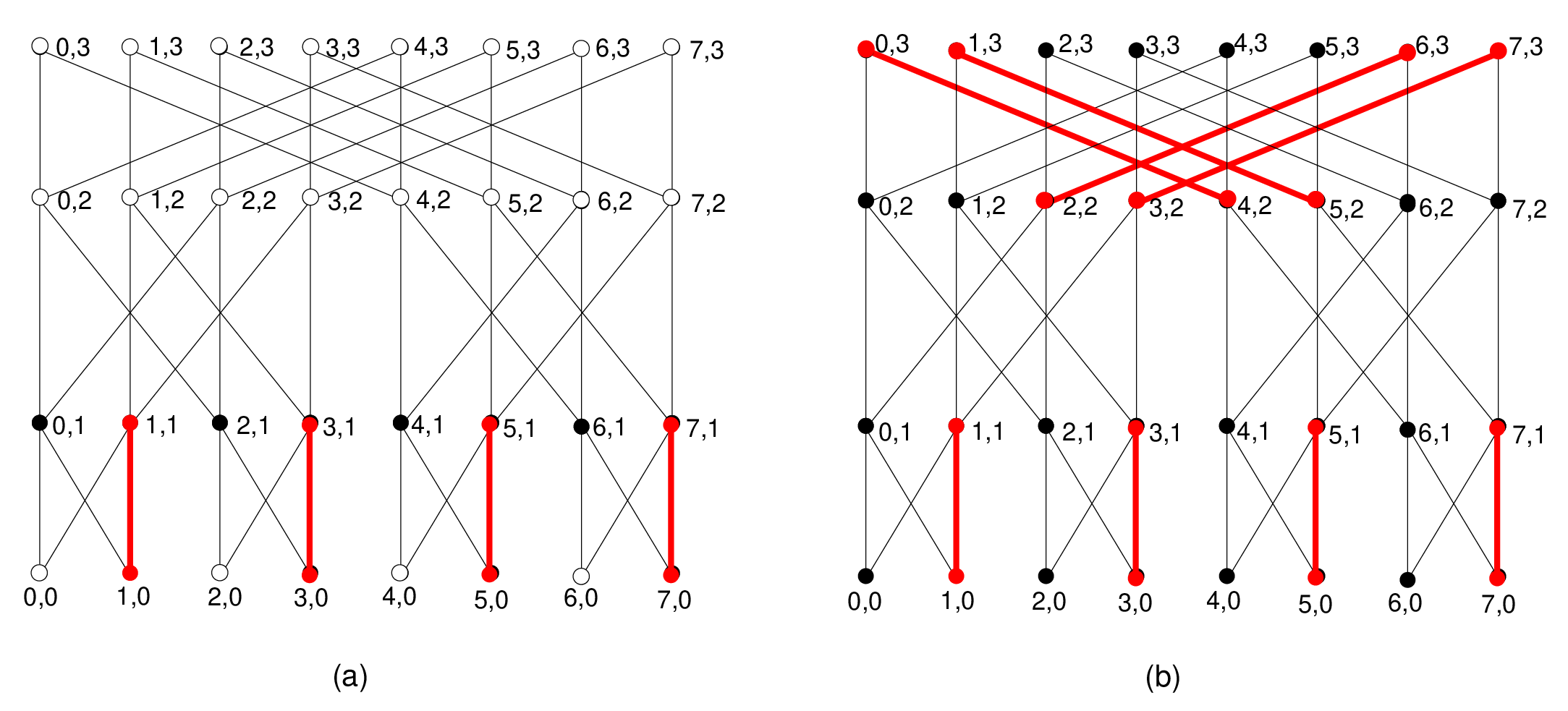}
\end{center}\vspace*{-3mm}
\includegraphics[scale=0.69]{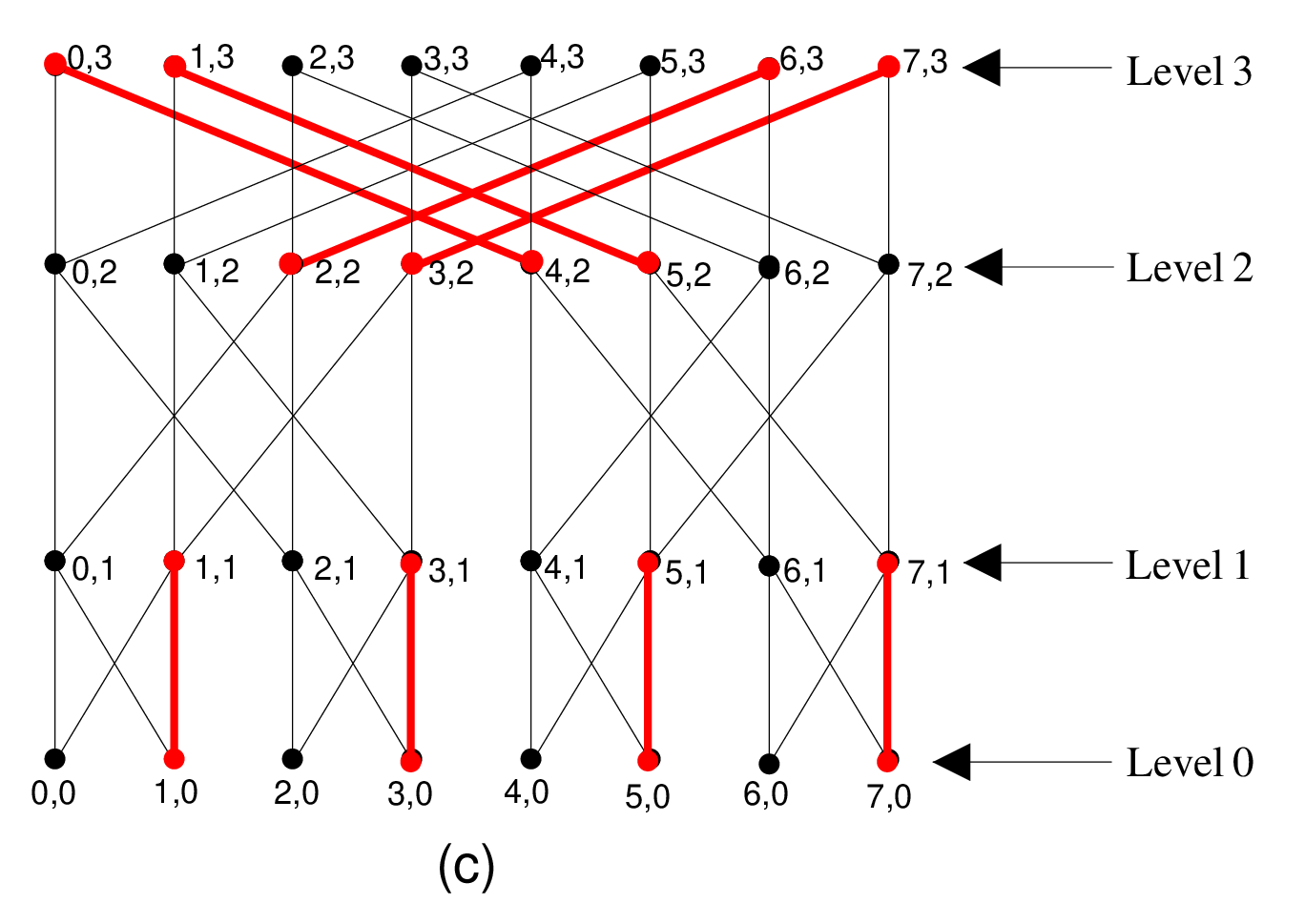}
\centering\vspace*{-5mm}
\caption{(a) $BF(3)$ Algorithm-Step 1, (b) $BF(3)$ Algorithm-Step 2, (c) Edge-forcing of $BF(3)$}
\label{figure5abc}
\end{figure}

\begin{flushleft} \textbf {\textit{Proof of Correctness}}:\end{flushleft}
The selected edges between Level 0 and Level 1 force all vertices of Level 1. See Figure~\ref{figure5abc}(a). The edges chosen between Levels 2 and 3 force all vertices of Level 2. See Figure~\ref{figure5abc}(b).  Thus in the first step, all vertices of Level 1 are forced. In the second step of the forcing process, the left out vertices of $BF(3)$ are forced. Figure~\ref{figure5abc}(c) illustrates the edge-forcing of $BF(3)$ and also the different Levels in a butterfly network.

By edge-forcing algorithm $BF(3)$, $\zeta_{e}(BF(3)) \leq 8$. Combining this with Lemma 4.2, we have the following result.

\begin{theorem}
The edge-forcing number of $BF(3)$ is $8$, that is $\zeta_{e}(BF(3)) = 8$.
\end{theorem}

\medskip
We next compute the edge-forcing number of $BF(4)$.
\begin{lemma}
The edge-forcing number of $BF(4)$ is at least $25$, that is $\zeta_{e}(BF(4))\geq 25$.
\end{lemma}

\begin{proof}There are 16 binding diamonds in $BF(4)$. 16 edges are chosen, one edge from every binding diamond as in the Lemma 4.2. This leaves all vertices in Levels 1 and 3 forced. Level 0 has 8 unforced vertices. The parallel edges connecting vertices of Level 1 and Level 2 are partitioned into 4 subsets $S_{0},S_{1},S_{2}$ and $S_{3}$ of 4 edges each, where end vertices of edges in $S_{i}$ are labelled  $([4i,1],[4i,2])$, $([4i+1,1],[4i+1,2])$, $([4i+2,1],[4i+2,2])$, $([4i+3,1],[4i+3,2])$, $i=0,1,2,3$. An edge chosen with one end in Level 1 and another end in Level 2 in any $S_{i}$, does not force any of the 8 unforced vertices in Level 0. But a pair of edges in $S_{i}$, $i=0,1,2,3$ forces 2 vertices of the corresponding vertices at Level 0. Thus to force vertices at Level 0, 8 edges are necessary in the edge-forcing set. There are vertices in Level 2 which are yet to be forced, hence at least one more edge has to be chosen in the edge-forcing set. Therefore  $\zeta_{e}(BF(4))\geq 25$.
\end{proof}

\begin{flushleft} \textbf {\textit{Edge-forcing Algorithm $BF(4)$}}:\end{flushleft}
\textbf{Input}: $BF(4)$, the butterfly network of dimension 4\\
\textbf{Algorithm}:\\
Step 1: Choose the edges $([2i-1,0],[2i-1,1]), 1 \leq i \leq 8$ and the edges $([i,4],[8i+1,3]), 0 \leq i \leq 8$, $([i,3],[8+i,4]), 4 \leq i \leq 7$;\\
Step 2: Choose the edges $([4i,1],[4i+4,2])$ and $([4i,2],[4i+2,1]), i=1,2$; \\
Step 3: Choose the edges $([i,2],[i,3]), i=0,1$ and edges $([i,3],[i,2]), i=14,15$; \\
Step 4: Select edge ([3,3], [7,2]). \\
\textbf{Output}: An edge-forcing set having cardinality 25

\begin{figure}[h!]
\includegraphics[scale=0.372]{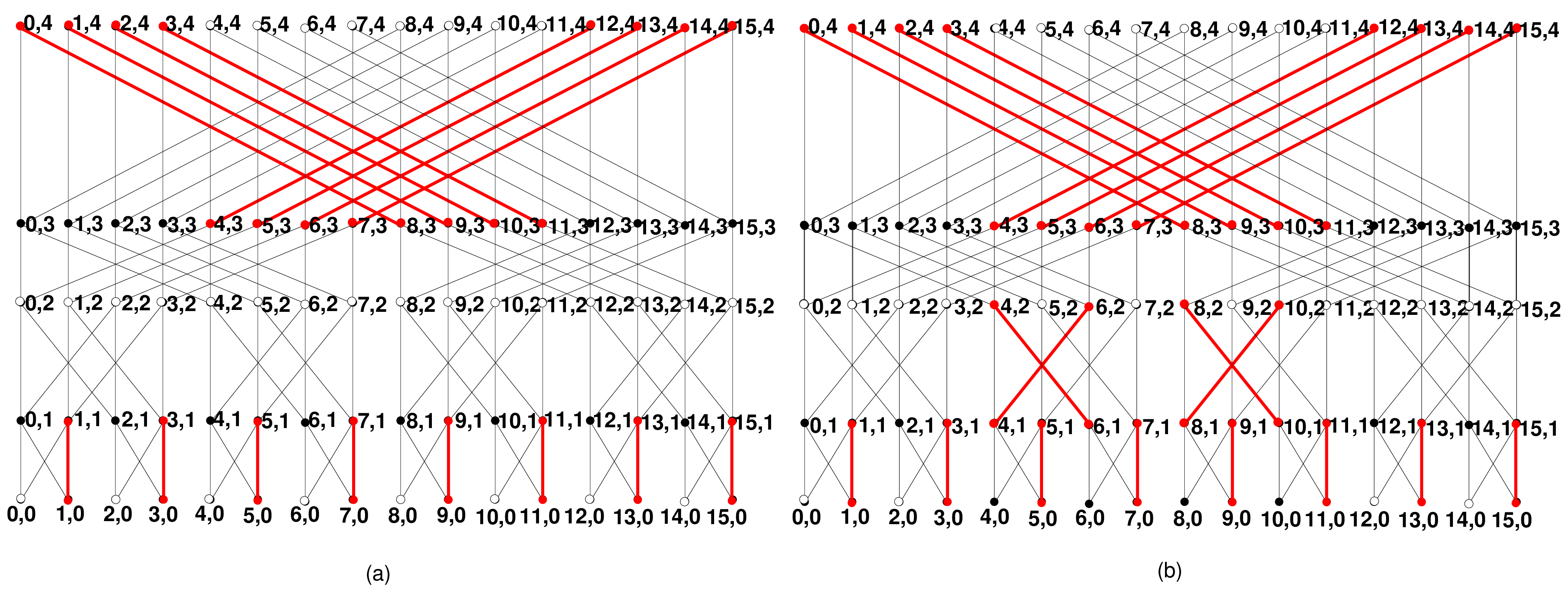}
\end{figure}

\begin{figure}[h!]
\includegraphics[scale=0.374]{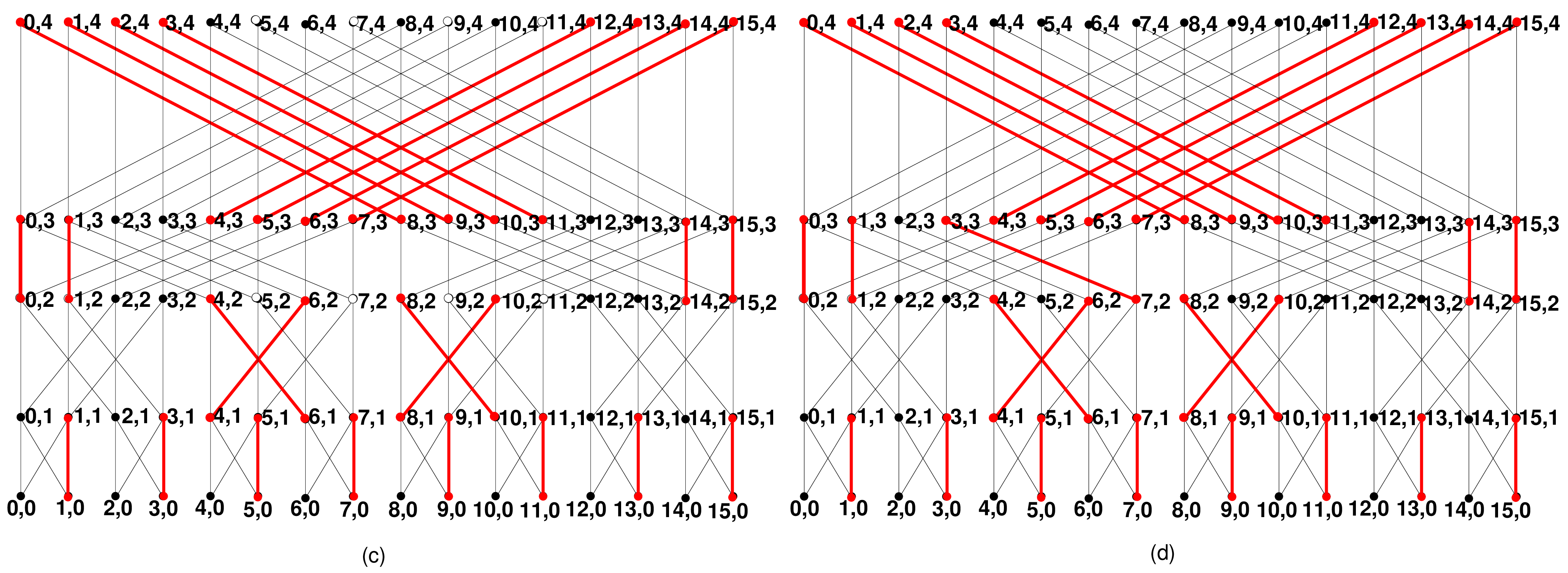}
\centering
\caption{(a) $BF(4)$ Algorithm-Step 1, (b) $BF(4)$ Algorithm-Step 2, (c) $BF(4)$ Algorithm-Step 3, (d) $BF(4)$ Algorithm-Step 4 and Edge-forcing of $BF(4)$ }
\label{figure6abcd}
\end{figure}

\begin{flushleft} \textbf {\textit{Proof of Correctness}}:\end{flushleft}
In Step 1, we have chosen 16 binding edges of $BF(4)$, one from each binding diamond. The edges selected in Step 1 force all vertices in Levels 1 and 3. See Figure~\ref{figure6abcd}(a). The edges selected using Step 2 and 3 force the remaining 8 vertices in Level 0. See Figure~\ref{figure6abcd}(b) and Figure~\ref{figure6abcd}(c). Any edge between the end vertices in Levels 2 and 3 will force all the remaining unforced vertices in Levels 2 and 4. See Figure~\ref{figure6abcd}(d).

\medskip
By edge-forcing algorithm $BF(4)$, $\zeta_{e}(BF(4)) \leq 25$. Combining this with Lemma 4.4, we have the following result.

\begin{theorem} The edge-forcing number of $BF(4)$ is $25$, that is $\zeta_{e}(BF(4)) = 25$.
\end{theorem}

\medskip
We now consider $BF(r)$, where $r=5$.

\begin{lemma} The edge-forcing number of $BF(5)$ is at least $47$, that is $\zeta_{e}(BF(5))\geq 47$.
\end{lemma}

\begin{figure}[h!]
\vspace{2mm}
\includegraphics[scale=0.376]{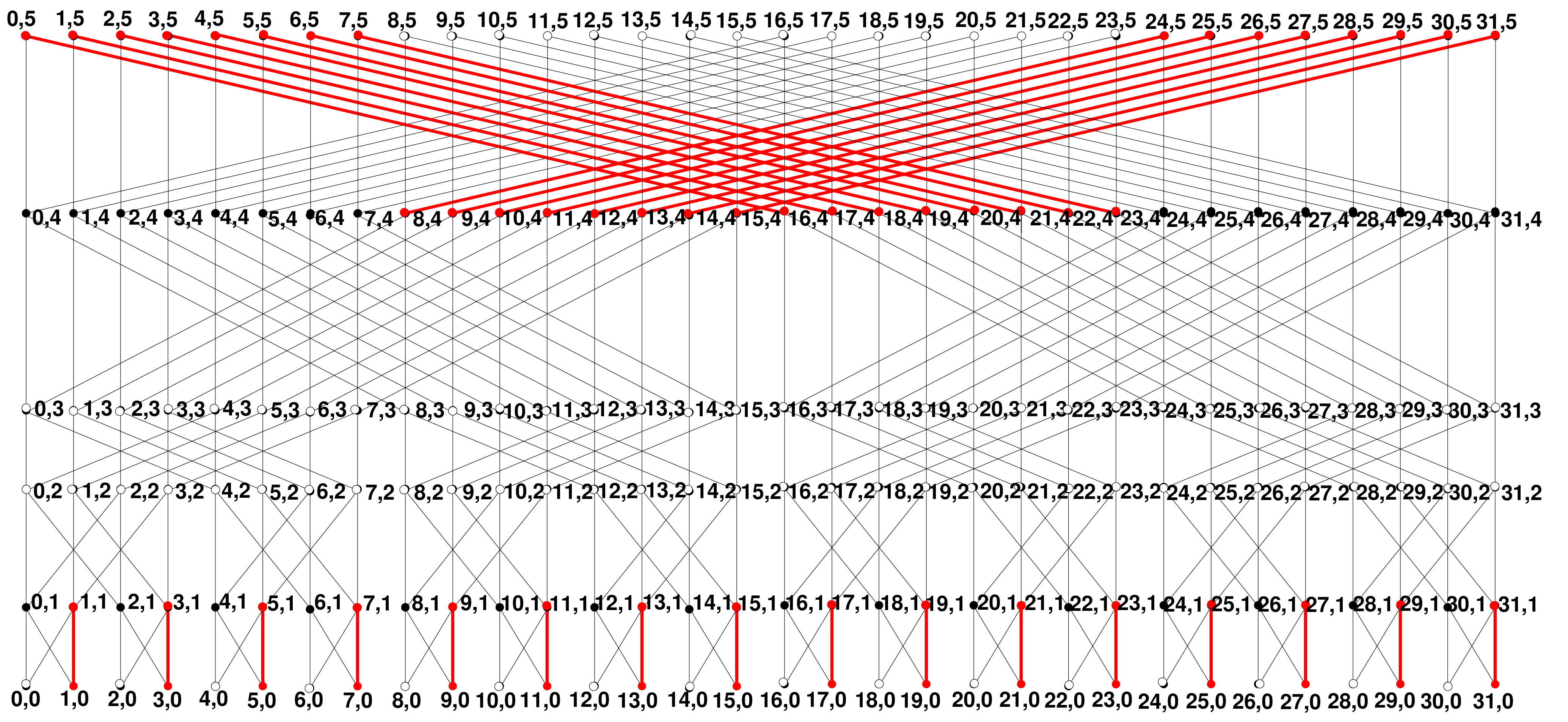}
\centering
\caption{$BF(5)$ Algorithm-Step 1}
\label{figure7}
\end{figure}

\begin{proof} There are 32 binding diamonds in $BF(5)$. 32 edges are chosen, one edge from every binding diamond as in Lemma 4.2. This leaves all vertices in Level 1 and Level 4 forced. Level 0 has 16 unforced vertices. The parallel edges connecting vertices of Level 1 and Level 2 are partitioned into 8 subsets $S_{0},S_{1},S_{2}$ and $S_{3}$,\ldots,$S_{7}$ of 4 edges each, where end vertices of edges in $S_{i}$ are labelled $([4i,1],[4i,2])$, $([4i+1,1],[4i+1,2])$, $([4i+2,1],[4i+2,2])$, $([4i+3,1],[4i+3,2])$, $i=0,1,2,3,\ldots,7$. Choosing an edge with one end in Level 1 and another end in Level 2 in any $S_{i}$, does not force any of the 8 unforced vertices in Level 0. But pairs of edges in $S_{i}$, $i=0,1,2,3,\ldots,6$ force 2 vertices of the corresponding vertices at Level 0. This accounts for an additional 12 edges in the edge-forcing set. Since all the vertices in the higher levels are already forced, only 3 edges are necessary from $S_{7}$ to force the remaining vertices. Therefore $\zeta_{e}(BF(5))\geq 47$.
\end{proof}

\begin{flushleft} \textbf {\textit{Edge-forcing Algorithm $BF(5)$}}:\end{flushleft}
\textbf{Input}: $BF(5)$, the butterfly network of dimension 5 \\
\textbf{Algorithm}:\\
Step 1: Choose the edges $([2i-1,0],[2i-1])$, $1\leq i \leq 16$ and edges $([i,5],[i+16,4])$ , $0 \leq i \leq 7$;  \\
Step 2: Choose the edges $([i+1,5],[i+1,3])$, $0 \leq i \leq 18$, $15 \leq i \leq 18$, $23 \leq i \leq 25$. \\
\textbf{Output}: An edge-forcing set having cardinality 47

\smallskip
\begin{flushleft} \textbf {\textit{Proof of Correctness}}:\end{flushleft}
In Step 1, we have chosen 32 binding edges of $BF(5)$, one from each of the binding diamonds. The edges selected using Step 1 force all the vertices in Level 1 and Level 4. See Figure~\ref{figure7}. The 15 edges chosen in Levels 2 and 3 by Step 2 are sufficient to force all vertices in Levels 2 and 3 and in turn all vertices in Levels 0 and 5. See Figure~\ref{figure8}.

\begin{figure}[h!]
\includegraphics[scale=0.376]{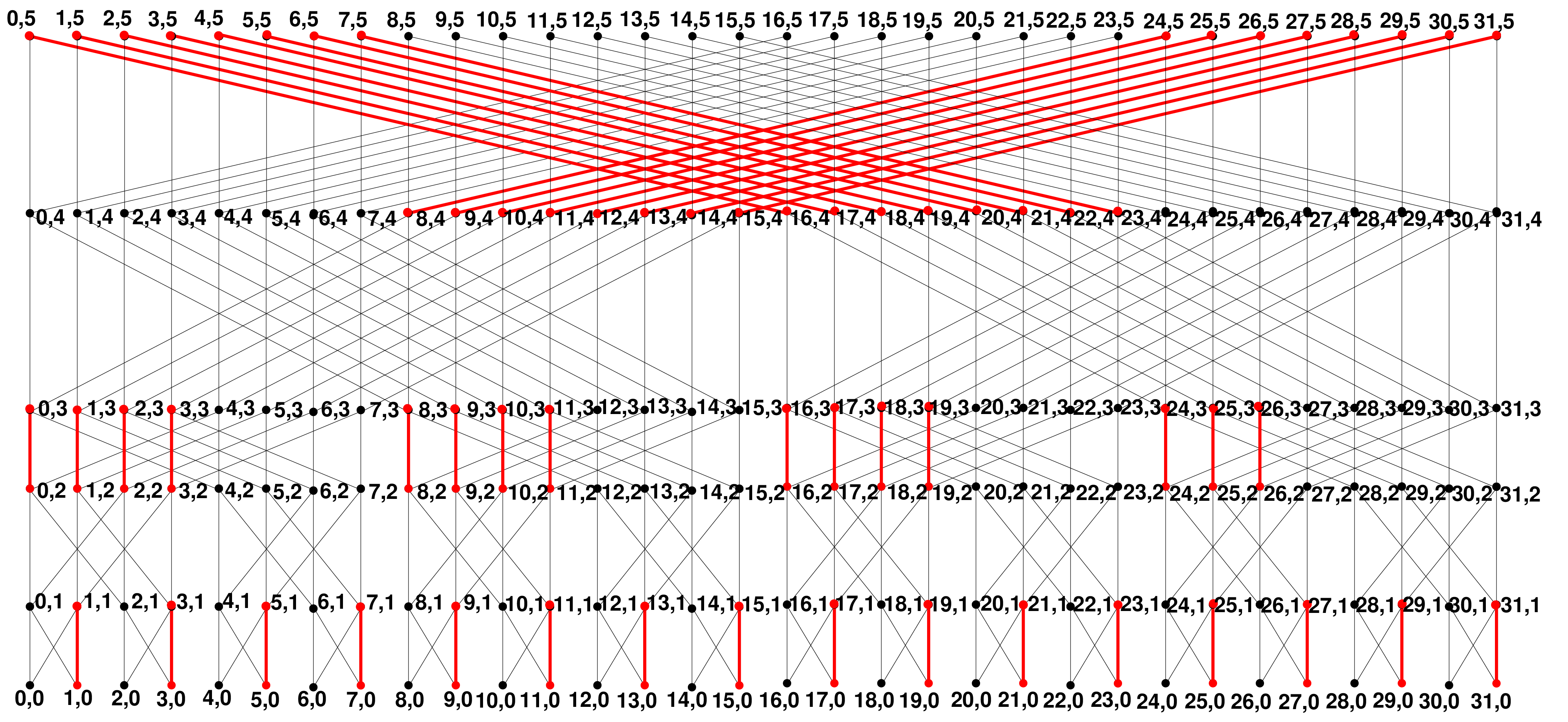}
\centering
\caption{$BF(5)$ Algorithm-Step 2 and Edge-forcing of $BF(5)$}
\label{figure8}
\end{figure}

\medskip
By edge-forcing algorithm $BF(5)$, $\zeta_{e}(BF(5)) \leq 47$. Combining this with Lemma 4.6, we have the following result.

\begin{theorem}
The edge-forcing number of $BF(5)$ is $47$, that is $\zeta_{e}(BF(5)) = 47$.
\end{theorem}

Now, we will prove the upper bound of the edge-forcing number of $BF(r)$, $r \geq 3$.

\begin{lemma}
For $r$ odd, $r \geq 3$, $\zeta_{e}(BF(r)) \leq \left\lceil \frac{r}{2} \right\rceil  2^{r-1}$.
\end{lemma}

\begin{proof}
The result is true for $r$ = 3 and 5 by Theorems 4.3 and 4.7 respectively. For $r \geq 7$, $r$ odd, $BF(r)$ contains 4 vertex disjoint copies of $BF(r-2)$ induced by vertices in Level 0 to Level $(r-2)$. The edges in the subgraph induced by vertices in Level $r$ and Level $(r-1)$ are not adjacent to any edge in the four isomorphic copies of $BF(r-2)$. Hence using recursion,
\begin{eqnarray*}
    \zeta_{e}(BF(r)) & \leq & 4 \zeta_{e}(BF(r-2)) + 2^{r-1} \\
    & \leq &  4 (4 \zeta_{e}(BF(r-4)) + 2^{r-3}) + 2^{r-1} \\
    & = & 4^2 \zeta_{e}(BF(r-4)) + 2\times 2^{r-1} \\
    & \leq & 4^{\left(\dfrac {r-3} {2}\right)}\times \zeta_{e}(BF(r-(r-3))) + \left({\dfrac {r-3} {2}}\right) 2^{r-1} \\
    & = & 2^{r-3} \times 8 + (r-3)~ 2^{r-2} \\
    & = & (r+1)~ 2^{r-2} \\
    & = & \left\lceil \frac{r}{2} \right\rceil  2^{r-1}.
\end{eqnarray*}

Thus $\zeta_{e}(BF(r)) \leq \left\lceil \frac{r}{2} \right\rceil  2^{r-1}; ~r \geq 3$ and $r$ odd.
\end{proof}

\noindent Combining Lemma 4.2 with Lemma 4.8, we arrive at the following result.

\begin{theorem} For $r \geq 3$, $r$ odd, $2^r \leq \zeta_{e}(BF(r)) \leq  \left\lceil \frac{r}{2} \right\rceil  2^{r-1}$.
\end{theorem}

\begin{lemma}For $r \geq 4$, $r$ even, $\zeta_{e}(BF(r)) \leq \left(\dfrac {r} {2} +2\right) 2^{r-1}$.
\end{lemma}

\begin{proof}
By Theorem 4.5, the result is true for $r$ = 4. For $r \geq 6$, $r$ even, $BF(r)$ contains four vertex disjoint copies of $BF(r-2)$ induced by vertices in Level 0 to Level $(r-2)$. The edges in the subgraph induced by vertices in Level $r$ and Level $(r-1)$ are not adjacent to any edge in the four isomorphic copies of $BF(r-2)$. Hence using recursion,\smallskip
\begin{eqnarray*}
    \zeta_{e}(BF(r))& \leq & 4\zeta_{e}(BF(r-2)) + 2^{r-1} \\
    & \leq &  4 (4 \zeta_{e}(BF(r-4)) + 2^{r-3}) + 2^{r-1} \\
    & = &  4^2 \zeta_{e}(BF(r-4)) + 2\times 2^{r-1} \\
    & \leq & 4^{\left(\dfrac {r-4} {2}\right)} \times \zeta_{e}(BF(r-(r-4))) + \left({\dfrac {r-4} {2}}\right) 2^{r-1} \\
    & = & 2^{r-4} \times 25 + (r-4)~ 2^{r-2} \\
    & = & \left({\dfrac {4r+9} {8}}\right) {2^{r-1}} \\
    & \leq & \left(\dfrac {r} {2} +2\right)~ 2^{r-1}.
\end{eqnarray*}

Thus $\zeta_{e}(BF(r)) \leq \left(\dfrac {r} {2} +2\right) 2^{r-1}; ~r \geq 4$ and $r$ even.
\end{proof}

 Combining Lemma 4.2 with Lemma 4.10, we arrive at the following result.

\begin{theorem}
For $r \geq 4$, $r$ even, $2^r \leq \zeta_{e}(BF(r)) \leq \left(\dfrac {r} {2} +2\right) 2^{r-1}$.
\end{theorem}

\section{Conclusion}
In this paper, we have introduced the edge-forcing problem in line with `total forcing', `connected forcing' and `$k$-forcing' studied by various authors. We have established the NP-completeness of the edge-forcing problem. We have obtained a lower bound for the edge-forcing number of butterfly networks $BF(r)$, $r \geq 3$. Further, we have proved that this lower bound is sharp for $BF(r)$, $r$ = 3, 4, 5 and have obtained an upper bound for higher dimensions. Determining edge-forcing number for higher dimensions of butterfly networks is a challenging open problem.

\subsection*{Acknowledgments}
The authors thank Mr. Andrew Arokiaraj for his insightful comments on the computational complexity of the problem. Further, the authors would like to thank the anonymous reviewers for their detailed comments which has helped us to improve the paper to a great extent.
The work of R. Sundara Rajan is partially supported by Project No. 2/48(4)/ 2016/NBHM-R\&D-II/11580, National Board of Higher Mathematics (NBHM), Department of Atomic Energy (DAE), Government of India.



\end{document}